\documentclass[sn-mathphys]{sn-jnl}

\usepackage{amsmath}
\usepackage{amssymb}

\usepackage{mathtools}

\usepackage{physics}
\usepackage{cleveref}

\usepackage{dsfont}

\usepackage{aligned-overset}

\usepackage{verbatim}

\usepackage{rotating}

\jyear{2021}%

\theoremstyle{thmstyleone}%
\newtheorem{theorem}{Theorem}
\newtheorem{proposition}[theorem]{Proposition}%
\newtheorem{lemma}[theorem]{Lemma} 
\newtheorem{corollary}[theorem]{Corollary} 
\newtheorem{remark}[theorem]{Remark} 

\theoremstyle{assumption}

\DeclareMathOperator*{\argmin}{arg\,min}

\newcommand{\Complex}{\mathbb{C}}
\renewcommand{\R}{\mathbb{R}}

\newcommand{\N}{\mathbb{N}}

\newcommand*{\defeq}{\mathrel{\vcenter{\baselineskip0.5ex \lineskiplimit0pt\hbox{\scriptsize.}\hbox{\scriptsize.}}} =}

\newcommand{\ts}{\textstyle}
\newcommand{\re}{\mathrm{Re}}

\def\BibTeX{{\rm B\kern-.05em{\sc i\kern-.025em b}\kern-.08em T\kern-.1667em\lower.7ex\hbox{E}\kern-.125emX}}

\begin{document}

\title[Wirtinger gradient descent methods for low-dose Poisson phase retrieval]{Wirtinger gradient descent methods for low-dose Poisson phase retrieval}


\author[1,2]{\fnm{Benedikt} \sur{Diederichs}}
\author[1,3]{\fnm{Frank} \sur{Filbir}}
\author*[1,3]{\fnm{Patricia} \sur{Römer} \email{patricia.roemer@tum.de}}
\affil[1]{\small  \orgdiv{Mathematical Imaging and Data Analysis, Institute of Biological and Medical Imaging}, \orgname{Helmholtz Center Munich}} 
\affil[2]{\small \orgdiv{Department of Chemistry and Center for NanoScience}, \orgname{Ludwig Maximilian University of Munich}}
\affil[3]{\small \orgdiv{Department of Mathematics}, \orgname{Technical University of Munich}}


\abstract{
The problem of phase retrieval has many applications in the field of optical imaging. Motivated by imaging experiments with biological specimens, we 
primarily consider the setting of low-dose illumination where Poisson noise plays the dominant role. In this paper, we discuss gradient descent algorithms 
based on different loss functions adapted to data affected by Poisson noise, in particular in the low-dose regime. Starting from the maximum log-likelihood function for the
Poisson distribution, we investigate different regularizations and approximations of the problem to design an algorithm that meets the requirements that are 
faced in applications. In the course of this, we focus on low-count measurements. For all suggested loss functions, we study the convergence of the 
respective gradient descent algorithms to stationary points and find constant step sizes that guarantee descent of the loss in each iteration. Numerical 
experiments in the low-dose regime are performed to corroborate the theoretical observations.
}





\maketitle

\section{Introduction} \label{sec:Intro}
Phase retrieval is a fundamental problem particularly in diffraction imaging techniques, where intensity measurements of several diffraction patterns of an 
object are recorded. In many instances, these measurements are taken in the far-field distance. It can be shown that such far-field intensity measurements are 
given by the squared absolute value of the Fourier transform of the object. To make this more concrete suppose a detector with $N$ pixels is placed in the far-field 
distance. A discretized form of the signal at the detector plane is then given by the discrete Fourier transform 
\begin{equation*}
\sum_{n=0}^{N-1} x_n  e^{-2\pi i\, k\, n/N} = \langle x,u_k\rangle,
\end{equation*}
with vectors $u_k=(e^{2\pi i\, k\, n/N})_{n=0}^{N-1}$ and $x=(x_n)_{n=0}^{N-1}$. The latter is a discrete representation of the object which we would like to reconstruct. 
According to this model, the vector $u_k$ can be related to the $k$-th frequency, or equivalently, with the $k$-th pixel of the detector. As mentioned above, the 
detector itself measures only intensities, that is, it records only the squared modulus of $\langle x, u_k\rangle$. Consequently, we are faced with the challenging 
problem of reconstructing $x$ from data of the form $\vert\langle x,u_k\rangle\vert^2,\ k=0,\dots, N-1$. 

It is known, however, that this data set does not contain enough information to make the inverse problem uniquely solvable \cite{beinert2015ambiguities}, i.e., it is not sufficient to probe the object $x$ 
only by the pure states $u_k, \ k=0,\dots, N-1$. To avoid these difficulties, we have to insert a certain amount of redundancy into our data set.  In order to introduce a sufficient portion  
of redundancy, one replaces the set of measurement vectors $u_k$ by an over-complete system of vectors $a_k\in\mathbb{C}^N, \ k=1,\dots, m$ with $m \gg N$, and measures 
$\vert\langle x, a_k\rangle\vert^2,\, k=1,\dots, m,$ instead of only $\vert\langle x,u_k\rangle\vert^2,\, k=0,\dots, N-1$. A prominent realization  for this approach is the so-called 
far-field ptychography. This method uses measurement vectors of the form $a_{k,\ell}=(w_{n,\ell}\,e^{-2\pi i\, k\, n/N})_{n=0}^{N-1}$ with a mask $w_{n,\ell}$, 
which is given as translates $(w_{n-\ell\, {\rm mod}\, N})_{n=0}^{N-1}$ of a vector $w=(w_n)_{n=0}^{N-1}$ with short support.
The number of translates is arranged such that the object is scanned by the mask in a way that for adjacent scanning positions the supports of the mask 
overlap with each other by a certain fraction. We will not discuss ptychography further, instead we refer to \cite{rodenburg2019ptychography, pfeiffer2018x, thibault2009probe}.

Motivated by ptychography, we consider the following  reconstruction problem. Given data of the form 
\begin{equation} \label{PR}
\hat{y}_{i} = \left\vert\left\langle x, a_{i}\right\rangle \right\vert^{2}, \quad i = 1,\ldots, m,
\end{equation}
we have to reconstruct the vector $x\in\mathbb{C}^n$. This is the phase retrieval problem in its mathematical abstract form. The problem attracted a lot of 
attention during the last decades and a number of fundamental contributions were made \cite{fienup1982phase, balan2006signal, candes2015phase_matrixcompletion, sun2018geometric}. The first fundamental problem is that of 
uniqueness. That is the question of how large $m$ has to be in order to make the mapping $\mathbb{C}^n\to\mathbb{C}^m,\ 
x\mapsto (\vert\langle x, a_i\rangle\vert^2)_{i=1}^m$ injective. This problem was considered by several authors under different assumptions on  $x$ and on the 
vectors $a_i$. We refer to the fundamental papers \cite{balan2006signal, conca2015algebraic} for further discussion. Our focus lies on the reconstruction of the 
vector $x$ given $\hat{y}_i,\, i=1,\dots, m$, and we will consider this problem under specific assumptions on the data set, which are again motivated by specific 
constraints in the experimental setup. Before going into the details of those, let us first recall the reconstruction techniques based on a variational approach. The 
variational method tries to recuperate $x$ via a minimization 
\begin{equation}\label{optimization_problem} 
x=\argmin_{z} \ \mathcal{L}(z), 
\end{equation}
using a suitable loss function $\mathcal{L}:\mathbb{C}^n\to\R$. For determining $x$, or at least a good approximation,
usually gradient descent methods are applied to the problem \eqref{optimization_problem}, i.e., 
\begin{equation*}
z_{k+1}=z_k-\mu\,\nabla\mathcal{L}(z_k),\ k\geq 0,
\end{equation*}
with learning rate $\mu > 0$ and some appropriate initial vector $z_0$. For the phase retrieval problem, the loss function is of the form $\mathcal{L}(z) = \ell\circ\varphi(z)$, 
with a  suitable function $\ell:\R^m\to\R$ and $\varphi:\Complex^n\to\R^m,\, z\mapsto (\vert\langle a_i, z\rangle\vert^2)_{i=1}^m$. Note also that the function $\varphi$ 
usually makes the problem non-convex.

A frequently used loss function for reconstruction of $x$ is the least squares loss 
\begin{equation*}
\mathcal{L}(z)=\sum_{i=1}^m\left\vert \vert\langle a_i,z\rangle\vert^2-\hat{y}_i \right\vert^2,    
\end{equation*}
or some regularized variants of it. Gradient descent algorithms for the corresponding minimization problem which apply Wirtinger calculus were 
investigated to some extent. These methods, now known as Wirtinger flow algorithms, were studied first by E. Cand{\' e}s, X. Li, and M. Soltanolkotabi in 
\cite{candes2015phase}. Different variants were further discussed by other authors \cite{zhang2017nonconvex, wang2017solving,chen2017solving,kolte2016phase, wang2017scalable, gao2020perturbed, tan2019phase, huang2022linear, romer2021randomized}. 
These works study the convergence of stochastic and non-stochastic gradient descent algorithms involving Wirtinger derivatives for different loss functions in case of random (Gaussian) 
measurement vectors $a_i$. Apart from convergence guarantees for random measurements, \cite{xu2018accelerated} and \cite{filbir2023image} analyze convergence of gradient descent algorithms to stationary points of an amplitude-based loss function for any type of measurement vectors, random or deterministic.

In all practical relevant measurement scenarios, the  data $\hat{y}_i$ is corrupted by some sort of noise. Hence, we are given 
perturbed values $y_i$ instead of $\hat{y}_i$. The perturbation may have many sources such as thermal noise, read-out noise, background noise, among others 
\cite{chang2019advanced}. The problem of a certain type of background noise was recently tackled by one of the authors in \cite{melnyk2024background}.  Here, 
we concentrate on the perturbation which is caused by the operation mode of the detector. All modern detectors, such as CCD cameras, are ultimately 
counting devices. That means, the measurement process can be modelled as a counting process and can, therefore, mathematically be formulated in 
terms of a Poisson distributed discrete random variable, viz. 
\begin{equation*}
y_{i} \sim \text{Poisson}\left(\left\vert\left\langle a_{i},x\right\rangle \right\vert^{2}\right), \quad i = 1,\ldots, m, 
\end{equation*}
where $\vert\langle a_i,x\rangle\vert^2$ is the ground-truth. This means that the probability that $y_i$ particles are counted at position $i$ 
given the ground-truth $\vert\langle a_i,x\rangle\vert^2$ is $\frac{1}{y_i!}\, e^{-\vert\langle a_i,x\rangle\vert^2}\, (\vert\langle a_i,x\rangle\vert^2)^{y_i}$. In order to incorporate the fact that the measurement process is a Poisson process into a variational reconstruction method, it is necessary to adapt the loss function accordingly. 
We follow a maximum likelihood approach and replace the least squares loss function by the Poisson log-likelihood loss which reads as 
$$
\mathcal{L}_P(z)=\sum_{i=1}^m \vert\langle a_i,z\rangle\vert^2-y_i\log(\vert\langle a_i,z\rangle\vert^2).
$$
The minimization problem \eqref{optimization_problem} for the Poisson likelihood loss $\mathcal{L}_P$ was studied by many authors, see for 
example \cite{thibault2012maximum, yeh2015experimental, hohage2016inverse, chen2017solving, chang2018total, fatima2022pdmm} and references therein. 

Our motivation to reconsider the minimization problem \eqref{optimization_problem} with Poisson log-likelihood $\mathcal{L}_P$ originates in specific diffraction 
imaging setups. For imaging biological tissue like cells, viruses, etc., it is generally not possible to work with a radiation beam of high intensity as such 
exposure would destroy the object instantaneously \cite{glaeser1971limitations, wang2009radiation, 2024lowdosecryo}. Therefore, it is essential to perform the measurement with a beam of suitably low intensity, 
which in turn leads to a weak signal at the detector. In realistic measurement scenarios, the counting rate can be in the range below $10$ illumination particles per pixel,  
and at a number of pixels it can be even zero. This low-count scenario leads to serious problems with respect to the 
gradient descent reconstruction process as it causes a singularity in the gradient of $\mathcal{L}_P$. In order to take these problems into account, we have to 
introduce a regularization technique which deals with these singularities. Algorithmic approaches for solving phase retrieval problems with low-dose Poisson noisy data were also considered in \cite{pelz2017low, godard2012noise, li2022poisson}.

The outline of the paper is as follows. In \Cref{WF_theory}, we summarize the theory on Wirtinger derivatives and some fundamental results on the 
convergence of the gradient descent algorithm with Wirtinger derivatives. \Cref{problem_formulation} contains a discussion on the choice of loss 
function in the optimization problem formulated for solving the phase retrieval problem with Poisson noisy low-dose data. In \Cref{WF_convergence}, we present a convergence analysis for the algorithms involving the different loss functions. We provide numerical justification of our theoretical results in \Cref{numerics}. In \Cref{conclusion}, we summarize our results and conclude with a brief outlook on further possible model adjustments.

\section{Gradient descent with Wirtinger derivatives} \label{WF_theory}

We start with presenting some theory on the Wirtinger calculus. Consider a function 
\begin{equation*}
f(z)=u(x,y)+i v(x,y),\ z=x+i y,\ x,y\in\R^n,    
\end{equation*}
with real-valued and 
differentiable functions $u$ and $v$. Using conjugate variables $z=x+i y$ and $\bar{z}=x-i y$, the function $f$ can be considered as a function of 
variables $z$ and $\bar{z}$. Since $u$ and $v$ are differentiable, the 
function $f(z,\bar{z})$ is holomorphic w.r.t. $z$ for fixed $\bar{z}$ and vice versa. The Wirtinger calculus expresses 
the derivatives of $f$ w.r.t. the real variables $x$ and $y$ in terms of the conjugate variables $z$ and $\bar{z}$ treating them as independent.  
The Wirtinger derivatives of $f$ are given as   
\begin{equation*}
\partial_z f={\ts\frac{1}{2}}\,(\partial_x f -i \partial_y f),\quad \partial_{\bar{z}} f={\ts\frac{1}{2}}\,(\partial_x f +i \partial_y f).
\end{equation*}
This implies the relation 
\begin{equation}\label{eq:Pre0a}
\overline{\partial_z f}=\partial_{\bar{z}}\bar{f},\quad\text{and}\quad \overline{\partial_{\bar z}f}=\partial_{z}\bar{f}.
\end{equation}
 The Wirtinger derivatives $\partial_zf$ and $\partial_{\bar{z}} f$ can also be expressed as 
\begin{align*}
\partial_z f
=\partial_z f(z,\bar{z})\big\vert_{\bar{z}=const.}=\begin{bmatrix} \partial_{z_1} f(z,\bar{z}),\dots, \partial_{z_n} f(z,\bar{z})\end{bmatrix}\big\vert_{\bar{z}=const.},\\[2ex]
\partial_{\bar{z}} f
=\partial_{\bar{z}} f(z,\bar{z})\big\vert_{z=const.}=\begin{bmatrix} \partial_{\bar{z}_1} f(z,\bar{z}),\dots, \partial_{\bar{z}_n} f(z,\bar{z})\end{bmatrix}
\big\vert_{z=const.}. 
\end{align*}
Consequently, the Wirtinger gradient and Wirtinger Hessian are given as 
\begin{equation}\label{eq:Pre1}
\nabla f(z)=\begin{pmatrix}(\partial_z f)^\ast\\ (\partial_{\bar{z}} f)^\ast\end{pmatrix},\quad 
\nabla^2f(z)=\begin{pmatrix}
                      \partial_{z}(\partial_z f)^\ast &\partial_{\bar{z}}(\partial_z f)^\ast\\[1ex]
                      \partial_z(\partial_{\bar{z}} f)^\ast&\partial_{\bar{z}}(\partial_{\bar{z}}f)^\ast
                     \end{pmatrix}.
\end{equation} 
In case that $f$ is a real-valued function, i.e., $f(z)=u(x,y)$, the relations in \eqref{eq:Pre0a} provide   
\begin{equation*}
\partial_{\bar{z}} f=\overline{\partial_z f},\quad \partial_{\bar{z}}(\partial_{\bar{z}}f)^\ast=\overline{\partial_z(\partial_z f)^\ast},\quad 
\partial_{z}(\partial_{\bar{z}} f)^\ast=\overline{\partial_{\bar{z}}(\partial_z f)^\ast}. 
\end{equation*}
It is more convenient to use the following simplified notation  
$$
\nabla_z f\defeq(\partial_z f)^\ast,\quad \nabla^2_{z,z} f\defeq\partial_{z}(\partial_z f)^\ast, \  \text{resp.} \quad \nabla^2_{z,\bar{z}} f\defeq\partial_{z}(\partial_{\bar{z}} f)^\ast.
$$
The second-order Taylor polynomial of $f$ at a point $z_0$ is 
$$
P_f(v,z_0)=f(z_0)+(\nabla f(z_0))^\ast \begin{pmatrix} v\\ \bar{v}\end{pmatrix}\, +\, \begin{pmatrix} v\\ \bar{v}\end{pmatrix}^\ast\, \nabla^2f(z)\, 
\begin{pmatrix} v\\ \bar{v}\end{pmatrix}.
$$
In case of  a real-valued function $f$, the quadratic term of the Taylor polynomial $P_f$ can be expressed as
\begin{equation}\label{eq:Pre2}
\begin{pmatrix} v\\ \bar{v}\end{pmatrix}^\ast\, \nabla^2f(z)\, \begin{pmatrix} v\\ \bar{v}\end{pmatrix}=2\re(v^\ast \nabla^2_{z,z}f(z)\, v)+2\re(v^\ast\, 
\nabla^2_{\bar{z},z}f(z)\, \bar{v}).
\end{equation}

In this paper, we focus on real-valued functions $f$. For minimizing such a function $f$, we apply gradient descent 
\begin{equation}\label{eq:Pre1a}
z_{k+1}=z_{k}-\mu_k\, \nabla_z f(z_{k}) 
\end{equation}
with some appropriate initial vector $z_0\in \Complex^n$. The parameter $\mu_k > 0$ is called step size or learning rate. It can be chosen to be constant or adaptive, preferably such that descent in every iteration is guaranteed, i.e., $f(z_{k+1})\leq f(z_k)$ for all $k \geq 0$. The proof of the following result can be found in \cite{filbir2023image}.

\begin{proposition} \label{proposition_stepsize_convergence}
Let $\ b \in \R, \ f: \Complex^{n} \rightarrow [b,\infty),$ be a twice Wirtinger differentiable function with a uniformly bounded Hessian, i.e., 
\begin{align*}
    \begin{pmatrix} v\\ \overline{v}
\end{pmatrix}^{*} \nabla^{2} f(z) \begin{pmatrix} v\\ \overline{v}
\end{pmatrix} \leq L \left\|  \begin{pmatrix} v\\ \overline{v} \end{pmatrix}  \right\|_{2}^{2}
 \end{align*}
for all $z,v \in \Complex^{n}$, with a constant $L > 0$ independent of $z$. Let the sequence $(z_{k})_{k\geq 0}$ be generated by the update \eqref{eq:Pre1a} with an arbitrary initialization $z_0 \in \Complex^{n}$. If $0 < \mu \leq L^{-1}$, then 
\begin{align*}
    f(z_k) - f(z_{k+1}) \geq \mu \left\|\nabla_{z} f (z_{k+1})\right\|_{2}^{2} 
\end{align*}
for all $k \geq 0$. 
Then, if $f$ has compact sublevel sets $L_{s}(f) = \left\{z \in \Complex^{n}: f(z) \leq s\right\}$, the Wirtinger flow algorithm \eqref{eq:Pre1a} is guaranteed to converge to a stationary point of $f$.
\end{proposition}

We now specify the requirements of this result to loss functions which are relevant in our context. As mentioned above, the loss functions we are going to consider are given as a 
composition $\ell\circ\varphi$ with a smooth function $\ell:\R\to\R$ and $\varphi:\Complex^n\to [0,\infty),\  \varphi(z)=\vert\langle a,z\rangle\vert^2$.  For such functions, we obtain the following bound on the Hessian.

\begin{lemma} \label{lemma_bound_hessian}
Let $\ell: \R \rightarrow \R$ be twice differentiable and $\varphi: \Complex^{n} \rightarrow [0,\infty), \ \varphi(z)=\left\vert\left\langle a,z\right\rangle\right\vert^{2}$ with $a \in \Complex^{n}$. Then 
\begin{align*}
 \begin{pmatrix} 
 v\\ \overline{v}
\end{pmatrix}^{*}
\nabla^{2}(\ell \circ \varphi)(z) 
\begin{pmatrix} 
v\\ \overline{v}
\end{pmatrix} 
\leq   \left(2 \ell''\bigl(\left\vert\left\langle a,z\right\rangle\right\vert^{2}\bigr) \left\vert\left\langle a,z\right\rangle\right\vert^{2} + \ell'\bigl(\left\vert\left\langle a,z\right\rangle\right\vert^{2}\bigr) \right) \left\|a\right\|_{2}^{2}\ \left\| \begin{pmatrix} v\\ \overline{v} \end{pmatrix}  \right\|_{2}^{2}
\end{align*}
for all $z,v \in \Complex^{n}$.
\end{lemma}

\begin{proof}[Proof.]
The Wirtinger derivative of  $(\ell \circ \varphi)(z,\overline{z}) = \ell(\overline{z}^{\mathrm{T}}aa^*z)$ is 
\begin{align*} 
\partial_z (\ell \circ \varphi)(z,\overline{z}) &= \partial_z(\ell \circ \varphi)(z,\overline{z})\mid_{\bar{z}=\text{const.}}\\  
&= \ell'(\overline{z}^{\mathrm{T}} aa^*z)\overline{z}^{\mathrm{T}}aa^{*} = \ell'(\left\vert\left\langle a,z\right\rangle\right\vert^{2})z^{*}aa^{*}.
\end{align*}
For the second derivatives we obtain 
\begin{align*} 
\partial_z\Big(\partial_z (\ell \circ \varphi)(z,\overline{z})\Big)^{*} &= \partial_z\Bigl(\ell'(\left\vert\left\langle a,z\right\rangle\right\vert^{2})aa^{*}z\Bigr)\\
&= \ell''(\left\vert\left\langle a,z\right\rangle\right\vert^{2})\left\vert\left\langle a,z\right\rangle\right\vert^{2} aa^{*} + \ell'(\left\vert\left\langle a,z\right\rangle\right\vert^{2})aa^{*},
\end{align*}
and analogously
\begin{align*}
\partial_{\bar{z}}\Big(\partial_z (\ell \circ \varphi)(z,\overline{z})\Big)^{*} = \partial_{\bar{z}}\Bigl(\ell'(\left\vert\left\langle a,z\right\rangle\right\vert^{2})aa^{*}z\Bigr)= \ell''(\left\vert\left\langle a,z\right\rangle\right\vert^{2}) \left\langle a,z\right\rangle^{2} aa^{\mathrm{T}}.
\end{align*}
In view of \eqref{eq:Pre2} we get 
\begin{align*} 
\begin{pmatrix}
v\vspace{1mm}\\
\overline{v}
\end{pmatrix}^{*} \nabla^{2}(\ell\circ \varphi)(z) \begin{pmatrix}
v\vspace{1mm}\\
\overline{v}
\end{pmatrix} &= 2 \Bigl( \ell''(\left\vert\left\langle a,z\right\rangle\right\vert^{2})\left\vert\left\langle a,z\right\rangle\right\vert^{2} + \ell'(\left\vert\left\langle a,z\right\rangle\right\vert^{2}) \Bigr) \left\vert\left\langle a,v\right\rangle\right\vert^{2}\\[4pt]
&\quad +  2 \operatorname{Re} \Bigl( \ell''(\left\vert\left\langle a,z\right\rangle\right\vert^{2}) \left\langle a,z\right\rangle^{2} \left\langle a,v\right\rangle^{2}\Bigr) \\[10pt]
& \leq 2  \Bigl( 2 \ell''(\left\vert\left\langle a,z\right\rangle\right\vert^{2})\left\vert\left\langle a,z\right\rangle\right\vert^{2} + \ell'(\left\vert\left\langle a,z\right\rangle\right\vert^{2}) \Bigr) \left\vert\left\langle a,v\right\rangle\right\vert^{2} \\[10pt]
& \leq \Bigl( 2 \ell''(\left\vert\left\langle a,z\right\rangle\right\vert^{2})\left\vert\left\langle a,z\right\rangle\right\vert^{2} + \ell'(\left\vert\left\langle a,z\right\rangle\right\vert^{2}) \Bigr) \left\|a\right\|_{2}^{2} \left\|\begin{pmatrix} v\\ \overline{v} \end{pmatrix} \right\|_{2}^{2},
\end{align*}
where we used  $\operatorname{Re}(\alpha^2\,\beta^2) \leq \left\vert \alpha \right\vert^{2}\left\vert \beta \right\vert^{2}$ for $\alpha,\beta \in \Complex$.
\end{proof}

In order to apply the convergence result of \Cref{proposition_stepsize_convergence}, we need to bound $2\,\ell''\left(x\right) \,x + \ell'\left(x\right) $
by a constant independent of $x\in[0,\infty)$ and we have to show that the level sets of $\ell\circ\varphi$ are compact. The compactness of the level sets will 
be addressed in the following lemma.  

\begin{lemma} \label{compact_levelsets}

    Let $f: \Complex^{n} \rightarrow \R, \ z \mapsto \sum_{i=1}^{m} \left\vert\left\langle a_{i},z\right\rangle\right\vert^{2}$ with $a_{i} \in \Complex^{n}, \ i = 1,\ldots,m$. Denote with $A$ the matrix with rows $a_{i}^{*}, \ i = 1,\ldots, m$. For any $z_{0} \in \Complex^{n}$, $f$ restricted to $z_{0} + \text{range}(A^{*})$ has compact level sets. 
    
\end{lemma}

\begin{proof}[Proof.]

    Consider any $z \in z_{0} + \text{range}(A^{*})$. We use that $\sum_{i=1}^{m} \left\vert\left\langle a_{i},z\right\rangle\right\vert^{2} = \left\|Az\right\|_{2}^{2}$.  It is $z_0 = \hat{z}_0 + \tilde{z}_0$ with $\hat{z}_0 \in \text{ker}(A)$ and $\tilde{z}_0 \in \text{range}(A^*)$.
     The vector $\tilde{z} \defeq z - \hat{z}_0$ is the orthogonal projection of $z$ onto the range of $A^*$. As $\tilde{z}$ is orthogonal to the kernel of $A$, it is $\left\|A\tilde{z}\right\|_{2}^{2} \geq \sigma_{\min}^2 \left\|\tilde{z}\right\|_{2}^{2}$, where $\sigma_{\min}$ denotes the smallest non-zero singular value of $A$.  Then, as $\left\|Az\right\|_{2}^{2} = \left\|A\tilde{z}\right\|_{2}^{2}$, we can continue to bound
\begin{equation*}
        \left\|Az\right\|_{2}^{2} \geq  \sigma_{\min}^2 \left\|\tilde{z}\right\|_{2}^{2}
        \geq \sigma_{\min}^2 \left( \left\|z\right\|_{2}^{2} - \left\|\hat{z}_0\right\|_{2}^{2}\right).
\end{equation*}
Since $\sigma_{\min}$ and $\left\|\hat{z}_0\right\|_{2}^{2}$ are constant for all $z \in z_{0} + \text{range}(A^{*})$, this shows that $\left\|Az\right\|_{2}^{2} $ is bounded from below by a scaled and shifted version of $\left\|z\right\|_{2}^{2}$ for all $z \in z_{0} + \text{range}(A^{*})$. As $z \mapsto \left\|z\right\|_{2}^{2}$ has compact level sets, $f$ has compact level sets on $z_{0} + \text{range}(A^{*})$.
\end{proof}

\begin{remark} \label{remark_compact_levelsets}
With \Cref{compact_levelsets} we obtain that $\sum_{i=1}^{m} \ell_{i}( \left\vert\left\langle a_{i},z\right\rangle\right\vert^{2})$ has compact level sets on $z_{0} + \text{range}(A^{*})$ if the functions $\ell_{i}:[0,\infty) \rightarrow \R, \ i = 1,\ldots,m$,  are continuous, bounded from below, and satisfy $\ell_{i}(t) \rightarrow \infty$ for $t \rightarrow \infty$.    
\end{remark}

\section{Phase retrieval as an optimization problem} \label{problem_formulation}

As discussed above, our aim is to reconstruct an (approximate) solution to the phase retrieval problem \eqref{PR} under certain assumptions on the random 
nature of the measurement process. Our focus lies on those cases where the measurement process can be modelled as a Poisson distributed random 
process. In order to associate this assumption with a variational reconstruction method, we use the maximum (log-)likelihood estimation.

\subsection{Poisson log-likelihood loss}\label{chapter: loss_poisson} 
In order to maximize the (log)-likelihood function for the Poisson distributed random variable $y_i$ with ground-truth $\vert\langle a_i,x\rangle\vert^2$, we have to 
determine $z$ such that $\sum_{i=1}^m (y_i\log(\vert\langle a_i,z\rangle\vert^2-\vert\langle a_i,z\rangle \vert^2)$ is maximal. This is equivalent to determining
$ \argmin_{z}\mathcal{L}_P(z)$ with 
\begin{equation}\label{loss_poisson}
\mathcal{L}_P(z)=\sum_{i=1}^m \vert\langle a_i,z\rangle\vert^2-y_i\, \log(\vert\langle a_i,z\rangle\vert^2).
\end{equation}
As mentioned earlier, we are interested in measurement scenarios where we can have no counts at certain pixels. In those cases we need to consider the value 
$\vert\langle a_i,z\rangle\vert^2=0$, which leads to singularities in the loss function $\mathcal{L}_P$.  A simple strategy to deal with them is 
shifting the logarithmic term by a positive constant $\varepsilon>0$, i.e.,
\begin{equation}\label{regularized_Poisson_loss}
\mathcal{L}_{P,\varepsilon}(z)=\sum_{i=1}^m \vert\langle a_i,z\rangle\vert^2-y_i\, \log(\vert\langle a_i,z\rangle\vert^2 + \varepsilon).
\end{equation}
The corresponding gradient descent update rule with constant step size reads as
\begin{equation*}
z_{k+1} = z_{k} - \mu\,  \sum_{i=1}^{m} \left(1 - \frac{y_{i}}{\left\vert\left\langle a_{i},z_{k}\right\rangle \right\vert^{2} + \varepsilon}\right) \left\langle a_{i},z_{k}\right\rangle a_{i}.
\end{equation*}

A suitable choice of the parameter is, however, not obvious. In this regard, one could for example study a discrepancy principle similar to \cite{bertero2010discrepancy, bevilacqua2021nearly}. We, instead, propose to consider an approximation of the Poisson loss \eqref{regularized_Poisson_loss}. An appropriate approximation is discussed in the next section.

\subsection{Gaussian log-likelihood and variance stabilization}\label{chapter: losses_variance_stabilization}

Recall that the maximum log-likelihood loss function for a Gaussian distributed random variable $y_i$ with distribution 
$\mathcal{N}(\vert\langle a_i,x\rangle\vert^2,\sigma_i^2),\, i=1,\dots, m,$ is given as 
\begin{equation} \label{loss_gauss} 
\mathcal{L}(z) = \sum_{i=1}^{m} \frac{1}{2\sigma_i^2} \left(\left\vert\left\langle a_{i},z\right\rangle \right\vert^{2} - y_{i} \right)^{2}.
\end{equation}
Often, this loss is used for image reconstruction even if the underlying process is a counting process, i.e., the random variable is Poisson distributed. This can 
be justified in cases for which the variance parameter $\lambda$ in the Poisson distribution $Poisson(\lambda)$ is sufficiently large. In this situation, the 
central limit theorem shows that $Poisson(\lambda)$ can be approximated by a  Gaussian distribution $\mathcal{N}(\lambda,\lambda)$. However, in the low-dose scenario this approximation is no longer suitable. 

Another approach to approximate a Poisson random variable by a Gaussian random variable is the so-called variance stabilization method 
\cite{tippett19352, bartlett1936square, curtiss1943transformations, anscombe1948transformation}. This method transforms the Poissonian data such that the resulting random variable has approximately constant 
variance. The rationale of this method is as follows. Let $\lambda>0$ and 
$X\sim Poisson(\lambda)$. Assume $f:\mathbb{R}_{\geq0}\to\mathbb{R}$ is a sufficiently smooth function. Then, its first order Taylor 
approximation around the variance $\lambda$ is $f(t)\approx f(\lambda)+(z-\lambda) f^\prime(\lambda)$, and, hence, $\mathbb{V}(f(X))\approx \mathbb{V}(X)\ (f^\prime(\lambda))^2.$ In order to obtain an approximate constant variance, set the right-hand side of the latter relation equal to $\sigma^2$. This leads to $f(t)=2\sigma\sqrt{t}$. The  
approach goes back to \cite{tippett19352, bartlett1936square} and was modified by Anscombe in \cite{anscombe1948transformation} using a 
fifth order Taylor approximation in order to handle cases where the variance $\lambda$ is rather small. Anscombe considered the shifted square-root 
transform $f(t)=\sqrt{t+c}$. Following the same arguments as above, one arrives at 
\begin{align*}
    \mathbb{V}\left(\sqrt{X + c} \, \right) &\approx (\lambda + c) \cdot  \mathbb{V} \Biggl[ 1 + \frac{1}{2} \cdot \frac{X - \lambda}{\sqrt{\lambda + c}} - \frac{1}{8}  \left(\frac{X-\lambda}{\sqrt{\lambda + c}}\right)^{2} + \frac{1}{16}  \left(\frac{X-\lambda}{\sqrt{\lambda + c}}\right)^{3}\\
    &\qquad\qquad\qquad\quad - \frac{5}{128}  \left(\frac{X-\lambda}{\sqrt{\lambda + c}}\right)^{4} + \frac{7}{256}  \left(\frac{X-\lambda}{\sqrt{\lambda + c}}\right)^{5}\Biggr]\\
    &\approx \frac{1}{4} \cdot \left( 1 + \frac{\tfrac{3}{8}-c}{\lambda} + \frac{32c^{2}-52c +17}{32\lambda^{2}}\right),  
\end{align*}
which suggests to choose $c = \frac{3}{8}$ to achieve $\mathbb{V}(X)\approx 1/4$. 

However, neither the simple square-root transform nor the Anscombe transform performs well when 
$\lambda\in [0,2]$, hence, we aim for a better suited variance-stabilizing transform of square-root type $f(t)=\sqrt{t+c}$. We try to determine the parameter $c>0$ such that $\mathbb{V}(\sqrt{X+c}\ )=1/4$ by considering 
\begin{align*}
\mathbb{V}\left(\sqrt{X+c} \ \right) &= \mathbb{E}(X) + c - \left[\mathbb{E}\left(\sqrt{X+c} \ \right)\right]^{2}\\[3pt]
 &= \lambda + c - \left( \sum_{k=0}^{\infty} \sqrt{k+c} \cdot \frac{\exp(-\lambda)\lambda^{k}}{k\text{!}} \right)^{2},
\end{align*}
and obtain an approximately optimal value for $c$ by setting this equal to $1/4$. By this method, we obtain, e.g., for $\lambda = 1$ the optimal value $c \approx 0.12$ and for $\lambda = 2$ the value $c \approx 0.27$.

To cover not only one optimal choice of $c$ for one specific value of $\lambda$, we propose to consider an averaging transform
\begin{equation} \label{averaging_transform}
    f(z) = \tfrac{1}{2}\left(\sqrt{z+c_{1}} + \sqrt{z + c_{2}} \ \right)
\end{equation}
with $c_{1}, c_{2} \geq 0$. An example for such a transform is the Tukey-Freeman transform \cite{freeman1950transformations}
\begin{equation} \label{tukey_freeman}
f(z) =  \tfrac{1}{2}\left(\sqrt{z} + \sqrt{z + 1} \ \right)
\end{equation}
that is known to perform well for small $\lambda$. We advance this idea with using $c_1 = 0.12$ and $c_2 = 0.27$ in an experiment dominated by $1$ and $2$ counts. This is justified as for measurements of this size it is very likely that the underlying ground-truth is close to $1$ or $2$. The performance of this transform compared to the square-root and the Anscombe transform can be studied in \Cref{fig:variance_stabilizing_transforms}.

\begin{figure}[!h]
\centering
\includegraphics[width=0.8\textwidth]
{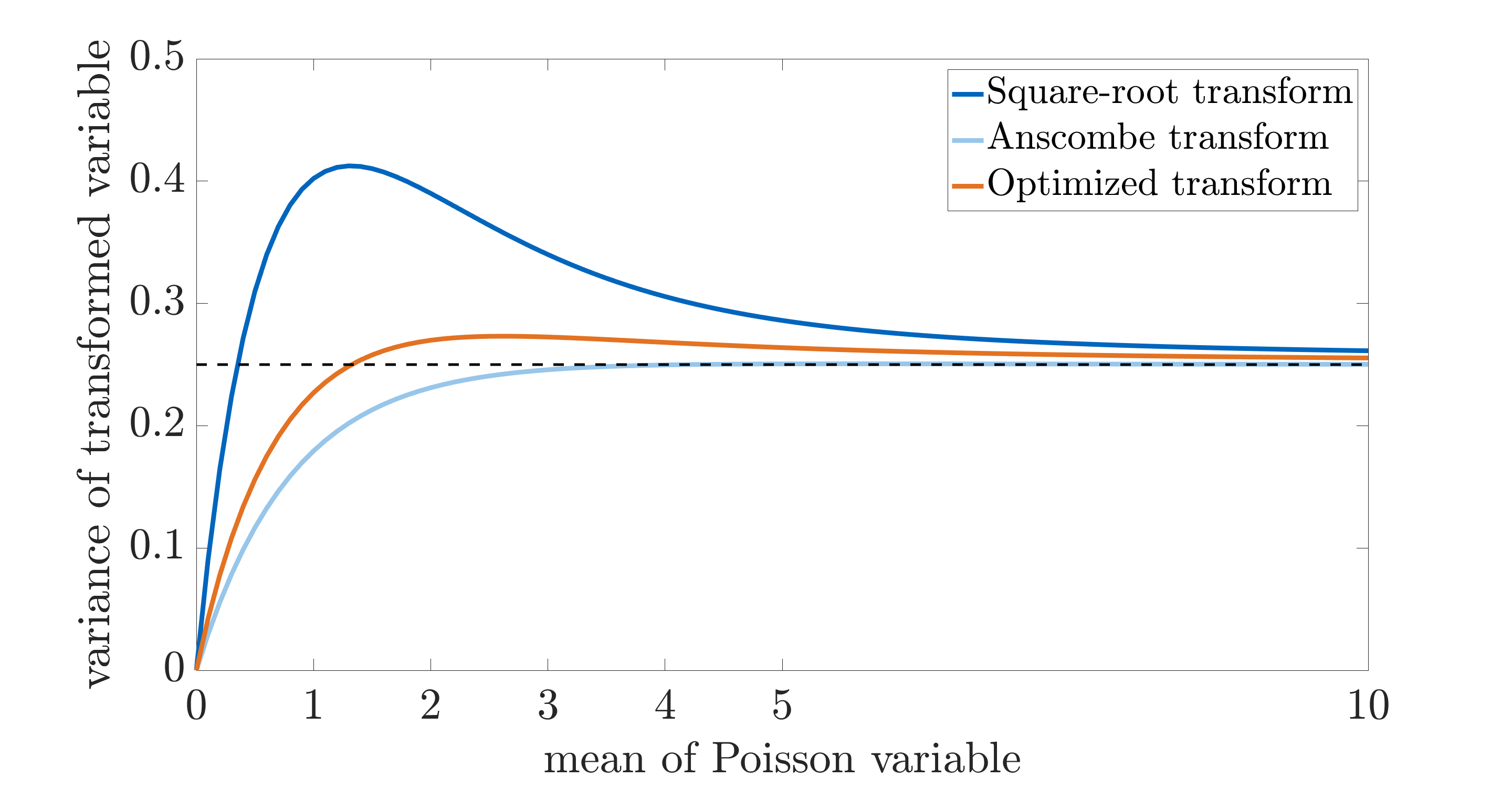}
\caption{Variance-stabilizing transforms.}
\label{fig:variance_stabilizing_transforms}
\end{figure}

By means of \Cref{fig:variance_stabilizing_transforms} we note that for $\lambda$ close to $0$ the square-root transform might be preferred over the other variance-stabilizing transforms. Also by the method described above we find that for $\lambda \rightarrow 0$ the optimal choice is $c$ close to $0$.

Returning back to the discussion on suitable loss functions, we can use any of the presented variance-stabilizing transforms $f$ and consider the loss 
\begin{equation*} 
\mathcal{L}(z) = \sum_{i=1}^{m} 2 \left( g\bigl(\left\vert\left\langle a_{i},z\right\rangle \right\vert^{2} \bigr) - f\left(y_{i}\right)\right)^{2},
\end{equation*}
which represents the log-likelihood function for a Gaussian distribution $\mathcal{N}\bigl(g\bigl(\left\vert\left\langle a_{i},z\right\rangle \right\vert^{2} \bigr), 1/4 \bigr)$, with function $g$ satisfying $g(X) = \mathbb{E}\left(f(X)\right)$.

If we consider a transform $f(x) = \sqrt{x + c}$ with $c \geq 0$ that stabilizes the variance of a Poisson random variable around $1/4$, the mean of the transformed random variable is approximately
\begin{align} \label{approximate_expectation}
    \mathbb{E}\bigl(\sqrt{X +c} \ \bigr) = \sqrt{\ \mathbb{E}\left(X + c\right) - \mathbb{V}\bigl(\sqrt{X+ c} \ \bigr) \ } \approx \sqrt{\ \mathbb{E}(X) + c - \tfrac{1}{4} \ }.
\end{align} 
Hence, we would work with a loss function
\begin{align} \label{loss_sqrt_c}
\mathcal{L}(z) = \sum_{i=1}^{m} 2 \left( \sqrt{\left\vert\left\langle a_{i},z\right\rangle \right\vert^{2} + c - \tfrac{1}{4}} - \sqrt{y_{i} + c}\right)^{2}.
\end{align}
A special case of this type of loss function is the amplitude loss
\begin{align} \label{amplitude_loss}
\mathcal{L}(z) = \sum_{i=1}^{m}2\left( \sqrt{\left\vert\left\langle a_{i},z\right\rangle \right\vert^{2} + \varepsilon} - \sqrt{y_{i}} \right)^{2},
\end{align}
with $\varepsilon > 0$, which has attracted interest recently in the phase retrieval community, see \cite{xu2018accelerated, filbir2023image}. This loss function is motivated by the Gaussian model involving a square-root transform. However, this formulation disregards the relation \eqref{approximate_expectation} but uses the coarse approximation $\mathbb{E}\bigl(\sqrt{X} \ \bigr) \approx \sqrt{ \ \mathbb{E}(X) \ }$.

A loss function based on the maximum likelihood model for a Gaussian distribution using the Anscombe transform, i.e., with $c = \frac{3}{8}$, was considered before, e.g., in \cite{konijnenberg2017study}, and studied for mixed Poisson-Gaussian noise in \cite{zhang2017fourier}. These works also used the approximation $\mathbb{E}\bigl(\sqrt{X + 3/8} \ \bigr) \approx \sqrt{  \mathbb{E}(X+3/8)  }$ instead of \eqref{approximate_expectation}.

While a loss function \eqref{loss_sqrt_c} fits very well the problem setting we want to consider, applying a gradient descent method to this loss function is problematic if $c \leq 1/4$. The considered function is not well-defined for $z$ with $\left\vert\left\langle a_{i},z\right\rangle \right\vert^{2} \in \left[0, 1/4 - c\right)$, and not differentiable at $\left\vert\left\langle a_{i},z\right\rangle \right\vert^{2} = 1/4 - c$. Hence, we need to regularize again to counteract this issue. We decide to also utilize the coarse approximation $\mathbb{E}\left(\sqrt{X + c} \ \right) \approx \sqrt{  \mathbb{E}(X+c)  }$ and consider
\begin{align} \label{loss_optimized_transform_regularized}
\mathcal{L}(z) = \sum_{i=1}^{m}2\left(  \sqrt{\left\vert\left\langle a_{i},z\right\rangle \right\vert^{2} + c}  - \sqrt{y_{i} + c } \right)^{2}.
\end{align}

The idea of adapting the expectation of the distribution accordingly is also not reflected in works as, e.g., \cite{gao2020perturbed}. The intention of those approaches is to smoothen the loss function. We, additionally, aim at formulating the loss such that it is close to the Poissonian model. However, in an experiment with too low illumination dose, resulting in ground-truth values much smaller than $1$, the variance after variance-stabilizing transform is rather close to $0$. Hence, the approximation \eqref{approximate_expectation} is not good for an extreme low-dose experiment and, thus, we abstain from subtracting $1/4$ in the mean approximation and prefer using \eqref{loss_optimized_transform_regularized}.

More generally, we can also use an averaging transform \eqref{averaging_transform},
with values $c_{1}, c_{2} \geq 0$, for variance stabilization. For the same reasons as before, we use the coarse approximation 
\begin{equation*}
\mathbb{E}\left(\sqrt{X + c_1} + \sqrt{X + c_2} \ \right) \approx \sqrt{\ \mathbb{E}(X) + c_1\ }  + \sqrt{\ \mathbb{E}(X) + c_2\ },
\end{equation*}
and consider loss functions of the form
\begin{equation*} 
\mathcal{L}(z) = \sum_{i=1}^{m}\frac{1}{2}\Bigl( \sqrt{ \ \left\vert\left\langle a_{i},z\right\rangle \right\vert^{2} + c_1 \ }  + \sqrt{ \ \left\vert\left\langle a_{i},z\right\rangle \right\vert^{2} + c_2 \ } - \sqrt{y_{i} + c_1 } -   \sqrt{y_{i} + c_2  } \ \Bigr)^{2}.
\end{equation*}

All discussed variance-stabilizing transforms result in a variance approximately equal to $1/4$ only for distribution parameters $\lambda > 0$. Hence, it might not be advisable to use the loss functions resulting from the variance-stabilizing transforms in case of ground-truth values equal or close to $0$. In an experiment, the ground-truth values are not known. However, if $y_{i} = 0$, it is likely that the corresponding ground-truth is close to $0$ and we wish to avoid variance-stabilizing transforms for these instances. Note that for $y_{i} = 0$ no regularization is required in the exact Poisson log-likelihood model as the loss function reduces to the term $\left\vert\left\langle a_{i},z\right\rangle \right\vert^{2}$. Therefore, we propose to consider 
\begin{align} \label{0_adaption} 
\mathcal{L}_0(z) = \sum_{i=1}^{m} & \ \mathds{1}_{y_{i} > 0} \cdot  \frac{1}{2}\left( \sqrt{ \ \left\vert\left\langle a_{i},z\right\rangle \right\vert^{2} + c_1\ }  + \sqrt{ \ \left\vert\left\langle a_{i},z\right\rangle \right\vert^{2} + c_2 \ } - C\ \right)^{2}\\[4pt] &+  \ \mathds{1}_{y_{i} = 0} \cdot \left\vert\left\langle a_{i},z\right\rangle \right\vert^{2}, \notag
\end{align}
with $c_2 \geq c_1 > 0, \ C \geq 0$, if one is working with low-dose data dominated by zero measurements.

\section{Convergence analysis of gradient descent algorithms for phase retrieval} \label{WF_convergence}

In the following, we analyze the convergence of gradient descent algorithms for optimization problems using the loss functions discussed in \Cref{chapter: loss_poisson} and \Cref{chapter: losses_variance_stabilization}.

Firstly, we state a convergence guarantee for a gradient descent algorithm according to the Poisson log-likelihood loss.

\begin{theorem} \label{theorem_convergence_poisson}
Let $\varepsilon > 0$ and $\mathcal{L}_{P,\varepsilon}(z) = \sum_{i=1}^{m} \left\vert\left\langle a_{i},z\right\rangle \right\vert^{2} - y_{i} \log\left(\left\vert\left\langle 
a_{i},z\right\rangle \right\vert^{2} + \varepsilon\right)$. Let $M$ be the matrix with rows $\sqrt{1+\tfrac{y_{i}}{8\varepsilon}} \cdot a_{i}^{*}, \  i = 1,\ldots, m$,
and choose  $\mu \leq  \left\| M \right\|^{-2} $. Then, the sequence $(z_{k})_{k \geq 0}\subset\Complex^n$ defined by
\begin{align*} 
z_{k+1} = z_{k} - \mu \cdot \sum_{i=1}^{m} \left(1 - \frac{y_{i}}{\left\vert\left\langle a_{i},z_{k}\right\rangle \right\vert^{2} + \varepsilon}\right) \left\langle a_{i},z_{k}\right\rangle a_{i}
\end{align*}
converges to a stationary point of $\mathcal{L}_{P,\varepsilon}$.
\end{theorem}
\begin{proof}[Proof.]
We shall apply \Cref{proposition_stepsize_convergence}. Since $\mathcal{L}_{P,\varepsilon}(z)=\sum_{i=1}^m\ell_i\circ\varphi(z)$ with 
$\ell_{i}: [0,\infty) \rightarrow \R, \ t \mapsto t - y_{i} \log(t + \varepsilon)$ and $\varphi(z)=\vert\langle a_i,z\rangle\vert^2$, we compute
\begin{align*}
    2 \cdot \ell_i''\left(t\right) \cdot t + \ell_i'\left(t\right) &= 2 \cdot \frac{y_{i}}{(t + \varepsilon)^{2}} \cdot t + 1 - \frac{y_{i}}{t + \varepsilon}\\[3pt]
    &= 1 + \frac{y_{i}(t - \varepsilon)}{(t + \varepsilon)^{2}} \, \leq \, 1 + \frac{y_i}{8\varepsilon},
\end{align*}
where we used that $\tfrac{s-\varepsilon}{(s+\varepsilon)^{2}}\leq \tfrac{1}{8\varepsilon}$ for every $s\in\mathbb{R}$ and any fixed $\varepsilon>0$. With \Cref{lemma_bound_hessian} and this bound, we obtain 
\begin{align*}
\begin{pmatrix} 
v\\ \overline{v}
\end{pmatrix}^{*} 
\nabla^{2}\mathcal{L}_{P,\varepsilon}(z) 
\begin{pmatrix} 
v\\ \overline{v}
\end{pmatrix} 
&\leq  2\cdot \sum_{i=1}^{m} \left\vert\left\langle \sqrt{1+\tfrac{y_{i}}{8 \varepsilon}} \cdot a_{i}, v\right\rangle\right\vert^{2}\\
&\leq 2\left\|M\right\|^{2} \left\|v\right\|_{2}^{2} = \left\|M\right\|^{2} 
   \left\|  \begin{pmatrix}
   v\vspace{2mm}\\ \overline{v}
   \end{pmatrix}  \right\|_{2}^{2}
\end{align*}
for all $z,v \in \Complex^{n}$. 
It remains to show that $\mathcal{L}_{P,\varepsilon}$ has compact sublevel sets on the subspace $z_{0} + \text{range}(A^{*})$, which contains all possibly attainable iterates of the algorithm. This follows from \Cref{compact_levelsets} and \Cref{remark_compact_levelsets}. Clearly, all functions $\ell_{i}$ are continuous and bounded from below. Further, 
we can find $k_1 > 0$ and $k_2 \in \R$ with $\ell_i(t) \geq k_1 \cdot t + k_2$ for all $t \in [0,\infty)$, hence $\ell_i(t) \rightarrow \infty$ for $t \rightarrow \infty$.
Then, using \Cref{proposition_stepsize_convergence}, we conclude that the considered gradient descent algorithm converges to a stationary point 
of the loss $\mathcal{L}_{P,\varepsilon}$.
\end{proof}

Our main conclusion of \Cref{theorem_convergence_poisson} is that the step size which guarantees convergence must be of the order of the regularization parameter $\varepsilon$. If we are interested in affecting the problem with only a small regularization parameter, this involves choosing a small step size, but using a constant small step size in all iterations results in slow convergence.

In \cite{chen2017solving}, a convergence analysis for a similar algorithm is presented. In contrast to our theory on convergence to stationary points, the authors of \cite{chen2017solving} state a guarantee for convergence to the ground-truth solution. However, while this guarantee only holds in case of a good initialization and is only applicable for measurement vectors $a_i$ being Gaussian random vectors, our result holds independently of the quality of the initialization and for arbitrary measurement systems and, hence, is more relevant for applications.

We proceed with a convergence analysis for the gradient descent algorithms based on the class of loss functions considered in \Cref{chapter: losses_variance_stabilization}.

\begin{theorem} \label{theorem_convergence_optimized_var_stab}
Let 
\begin{align} \label{loss_averaging_transform}
\mathcal{L}_{avg}(z) = \sum_{i=1}^{m}\frac{1}{2}\left( \sqrt{ \ \left\vert\left\langle a_{i},z\right\rangle \right\vert^{2} + c_{1}\ }  + \sqrt{\ \left\vert\left\langle a_{i},z\right\rangle \right\vert^{2} + c_{2} \ } - C \right)^{2},
\end{align}
with  constants $c_2 \geq c_1 > 0, \ C \geq 0$. For arbitrary $z_0\in\mathbb{C}^n$ define a sequence $(z_k)_{k\geq 1}$ by 
\begin{align*} 
z_{k+1} = z_{k} - \mu \cdot \frac{1}{2}\sum_{i=1}^{m} &\left( \sqrt{ \ \left\vert\left\langle a_{i},z_k\right\rangle \right\vert^{2} + c_{1}\ }  + \sqrt{ \ \left\vert\left\langle a_{i},z_k\right\rangle \right\vert^{2} + c_{2} } - C \right)\\[2pt]
& \cdot \left( \frac{1}{ \sqrt{\left\vert\left\langle a_{i},z_k\right\rangle \right\vert^{2} + c_{1}}}  + \frac{1}{\sqrt{\left\vert\left\langle a_{i},z_k\right\rangle \right\vert^{2} + c_{2} \ }}  \right) \left\langle a_{i},z_k\right\rangle a_{i}.
\end{align*}
Then,  $(z_{k})_{k \geq 0}$ converges to a stationary point of the loss function $\mathcal{L}_{avg}$ 
provided $\mu \leq 2 \left(\bigl(3 + \sqrt{c_2/ c_1}\bigr) \left\|A\right\|^{2} \right)^{-1}$, where $A$ denotes the matrix with rows $a_{i}^{*}, \  i = 1,\ldots, m$.
\end{theorem}
\begin{proof}[Proof.]
We proceed as in the proof of \Cref{theorem_convergence_poisson}. A bound for the Hessian of $\mathcal{L}_{avg}$ is found by using \Cref{lemma_bound_hessian} with $\ell_i: [0,\infty) \rightarrow [0,\infty), \  t \mapsto \frac{1}{2}\left( \sqrt{t + c_{1}}  + \sqrt{t + c_{2} } - C \right)^{2}$. We have 
\begin{align*}
    \ell'(t) &= \frac{1}{2}\left(\sqrt{t + c_1} + \sqrt{t+c_2} -C\right) \cdot \left(\frac{1}{\sqrt{t+c_1}} + \frac{1}{\sqrt{t+c_2}} \right)\\[6pt]
    &= \frac{1}{2}\left[2 + \frac{\sqrt{t + c_1}}{\sqrt{t + c_2}} + \frac{\sqrt{t + c_2}}{\sqrt{t + c_1}} - C\cdot\left(\frac{1}{\sqrt{t+c_1}}+\frac{1}{\sqrt{t+c_2}}\right)\right],
\end{align*}
which leads to 
 \begin{align*}
    \ell''(t) &= \frac{1}{4}\left(\frac{1}{\sqrt{t+c_1}} + \frac{1}{\sqrt{t+c_2}} \right)^2\\[2pt] & \quad - \frac{1}{4} \left(\sqrt{t + c_1} + \sqrt{t+c_2} -C\right) \cdot \left(\frac{1}{\left(t+c_1\right)^{\frac{3}{2}}} +\frac{1}{\left(t+c_2\right)^{\frac{3}{2}}}\right)  \\[6pt]
    &= \frac{1}{4} \Biggl[\frac{2}{\sqrt{t+c_1}\sqrt{t+c_2}} - \frac{\sqrt{t + c_1}}{\left(t+c_2\right)^{\frac{3}{2}}} - \frac{\sqrt{t+c_2}}{\left(t+c_1\right)^{\frac{3}{2}}}\\[2pt] 
    &\quad+ C \cdot \left(\frac{1}{\left(t+c_1\right)^{\frac{3}{2}}} +\frac{1}{\left(t+c_2\right)^{\frac{3}{2}}}\right)\Biggr] \\[6pt]
    &\leq  \frac{1}{4}  C \cdot \left(\frac{1}{\left(t+c_1\right)^{\frac{3}{2}}} +\frac{1}{\left(t+c_2\right)^{\frac{3}{2}}}\right).
\end{align*} 
For the bound of the second derivative we have used that $\frac{2}{\alpha\beta} \leq \frac{\alpha}{\beta^3} + \frac{\beta}{\alpha^3}$ for $\alpha,\beta \in \R$.  
Thus, we arrive at  
\begin{align*}
    2\, t\, \ell''\left(t\right) + \ell'\left(t\right) &\leq  \frac{1}{2}  C \cdot \left( \frac{t}{\left(t+c_1\right)^{\frac{3}{2}}} +\frac{t}{\left(t+c_2\right)^{\frac{3}{2}}} \right)\\[6pt] 
& \quad + \frac{1}{2}\left[2 + \frac{\sqrt{t + c_1}}{\sqrt{t + c_2}} + \frac{\sqrt{t + c_2}}{\sqrt{t + c_1}} - C\cdot\left(\frac{1}{\sqrt{t+c_1}}+\frac{1}{\sqrt{t+c_2}}\right)\right]\\[6pt]
    &= \frac{1}{2}\left( - C \cdot  \left(\frac{c_1}{\sqrt{t+c_1}}+\frac{c_2}{\sqrt{t+c_2}}\right)  +  2 + 1 + \sqrt{\frac{c_2}{c_1}} \ \right)\\[6pt]
    &\leq \frac{1}{2}\left(3 +  \sqrt{\frac{c_2}{c_1}} \ \right),
\end{align*}
and, consequently,
\begin{align*}
        \begin{pmatrix} v\\ \overline{v}
\end{pmatrix}^{*} \nabla^{2}\mathcal{L}_{avg}(z) \begin{pmatrix} v\\ \overline{v}
\end{pmatrix} \leq 
\sum_{i=1}^{m} &\Biggl(3 + \sqrt{\frac{c_2}{c_1}} \ \Biggr)\left\vert\left\langle a_{i},v\right\rangle\right\vert^{2}\\[6pt]
\leq 
\frac{1}{2} \Biggl(&3 + \sqrt{\frac{c_2}{c_1}} \ \Biggr) \left\|A\right\|^{2} \left\|  \begin{pmatrix}
v\vspace{2mm}\\
\overline{v}
\end{pmatrix}  \right\|_{2}^{2}
\end{align*}
for all $z,v \in \Complex^{n}$. This suggests choosing $\mu =2 \left(\bigl(3 +  \sqrt{c_2/ c_1}\bigr) \left\|A\right\|^{2}\right)^{-1}$. 

The function $\mathcal{L}_{avg}$ restricted to $z_{0} + \text{range}(A^{*})$ has compact level sets by \Cref{compact_levelsets} and \Cref{remark_compact_levelsets}, as all functions 
$\ell_{i}$ are continuous, bounded from below, and satisfy $\ell_i(t) \rightarrow \infty$ for $t \rightarrow \infty$.
By \Cref{proposition_stepsize_convergence}, the  gradient descent algorithm converges to a stationary point of the loss function 
$\mathcal{L}_{avg}$.
\end{proof}

By this result, we also obtain a rule for the step size when using a gradient descent algorithm for a loss \eqref{loss_sqrt_c} or \eqref{amplitude_loss}.

\begin{corollary}
Convergence to a stationary point of a loss function \eqref{amplitude_loss} for some regularization parameter $\varepsilon > 0$ can be achieved by the corresponding gradient descent algorithm using a step size $\mu \leq \left(2 \left\|A\right\|^{2} \right)^{-1}$.
\end{corollary}

This is the same result as Theorem A.1 in \cite{xu2018accelerated}, which means that we generalized the convergence guarantee of \cite{xu2018accelerated} for the amplitude loss to the class of loss functions \eqref{loss_averaging_transform}.

The important difference of this result to the convergence guarantee for the Poisson log-likelihood loss, as stated in \Cref{theorem_convergence_poisson}, is that 
the bound on the step size necessary for convergence is independent of a potentially small regularization parameter.

\begin{corollary}
A loss function using the Tukey-Freeman transform \eqref{tukey_freeman} for variance stabilization would be of the form
\begin{align*} 
\mathcal{L}(z) = \sum_{i=1}^{m}\frac{1}{2}\left( \sqrt{ \ \left\vert\left\langle a_{i},z\right\rangle \right\vert^{2} + \varepsilon\ }  + \sqrt{ \ \left\vert\left\langle a_{i},z\right\rangle \right\vert^{2} + 1 \ } - \sqrt{y_{i}} - \sqrt{y_{i} + 1} \right)^{2},
\end{align*}
with $\varepsilon > 0$ to avoid the singularity in the derivative. By \Cref{theorem_convergence_optimized_var_stab}, we conclude that a gradient descent algorithm using this loss function requires a step size $\mu \leq 2 \left(\bigl(3 + \sqrt{1/ \varepsilon}\bigr) \left\|A\right\|^{2} \right)^{-1}$ for guaranteed convergence.
\end{corollary}

This means that for this type of variance-stabilizing transform we face the same problematic as for the Poisson log-likelihood loss.

Also for the loss $\mathcal{L}_0$ finally proposed by us we can use the same fixed step size and can guarantee convergence to a stationary point of the loss function.

\begin{corollary}
Convergence to a stationary point of loss $\mathcal{L}_0$ defined in \eqref{0_adaption} can be achieved by the corresponding gradient descent algorithm with a step size $\mu \leq 2 \left(\bigl(3 + \sqrt{c_2/ c_1}\bigr) \left\|A\right\|^{2} \right)^{-1}.$
\end{corollary}

\section{Numerical experiments} \label{numerics}

\subsection{Phase retrieval from low-dose data with Poisson noise}

We corroborate our theoretical consideration with numerical experiments on the reconstruction of an object from simulated low-dose Poissonian phase retrieval measurements. 
Our simulations are based on a test object $x \in \Complex^{n}$ with $n = 256$, and complex Gaussian measurement vectors $a_{i} \in \mathbb{C}^{n}, \ i = 1,\ldots,m$, where $m = 10n$. We normalize the measurements such that $\sum_{i=1}^{m} \left\vert\left\langle a_{i},x \right\rangle\right\vert^{2} = 1$ and understand a noiseless value $\left\vert\left\langle a_{i},x \right\rangle\right\vert^{2}$ as the probability that photons arrive at the $i$-th detector pixel. We choose a dose $d \in \N$ and work with measurements
\begin{equation*}
    y_{i} \sim \text{Poisson}\left(d \cdot \left\vert\left\langle a_{i},x\right\rangle \right\vert^{2}\right), \quad  i = 1,\ldots, m.
\end{equation*}
In the following, we experiment with doses $d \in \left\{500, 1000, 1500, \ldots,  4000\right\}$. The histograms in \Cref{histogram} show exemplary realizations of measurement distributions for $d = 500$ and $d = 4000$. The lowest dose corresponds to a signal-to-noise-ratio
$SNR \defeq \|(d \cdot \vert\left\langle a_{i},x\right\rangle\vert^2 )_{i=1}^m \|_{2} \ \big/ \ \| ( y_{i} - d \cdot \left\vert\left\langle a_{i},x\right\rangle\right\vert^{2} )_{i=1}^m \|_{2}$
of approximately $0.6$, the highest dose corresponds to $SNR \approx 1.7$.

\begin{figure}[!h]
\centering
\begin{minipage}[b]{0.45\textwidth}
\centering
  \includegraphics[width=0.75\textwidth]{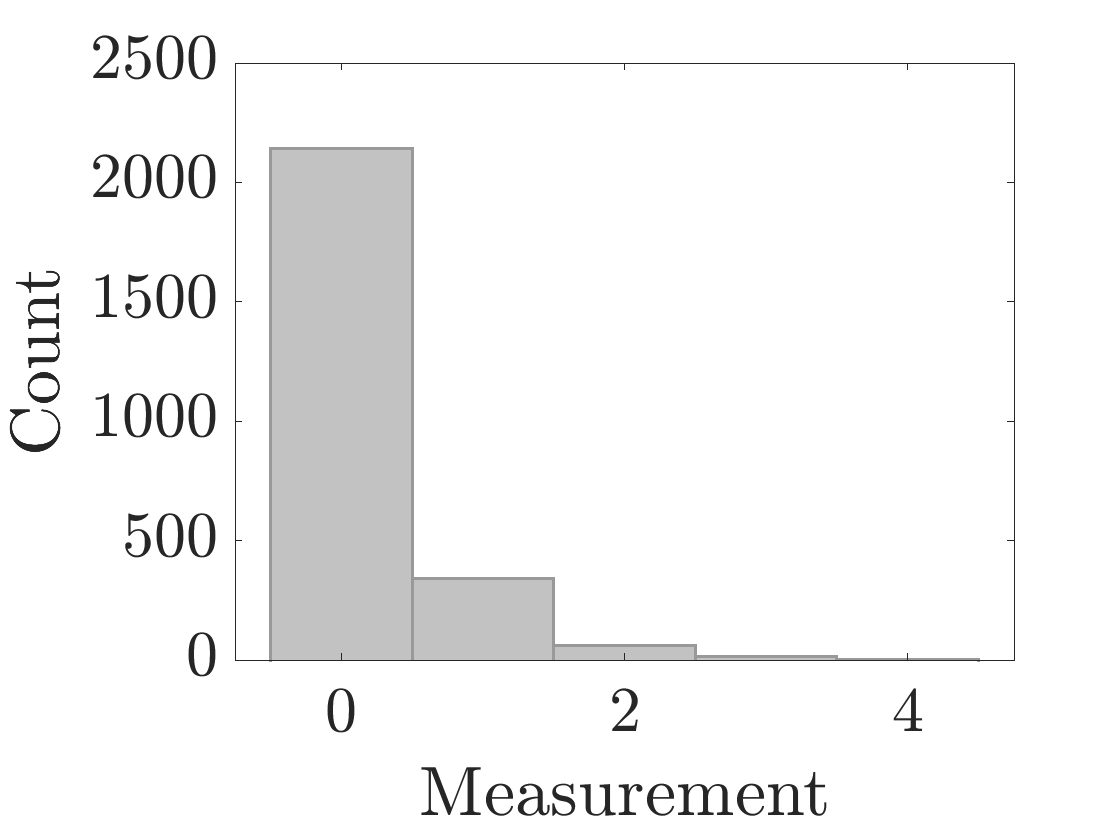}\\
\figurecaptionfont (a) Dose $d = 500$.
\end{minipage}
\hspace{10pt}
\begin{minipage}
[b]{0.45\textwidth}
\centering
  \includegraphics[width=0.75\textwidth]{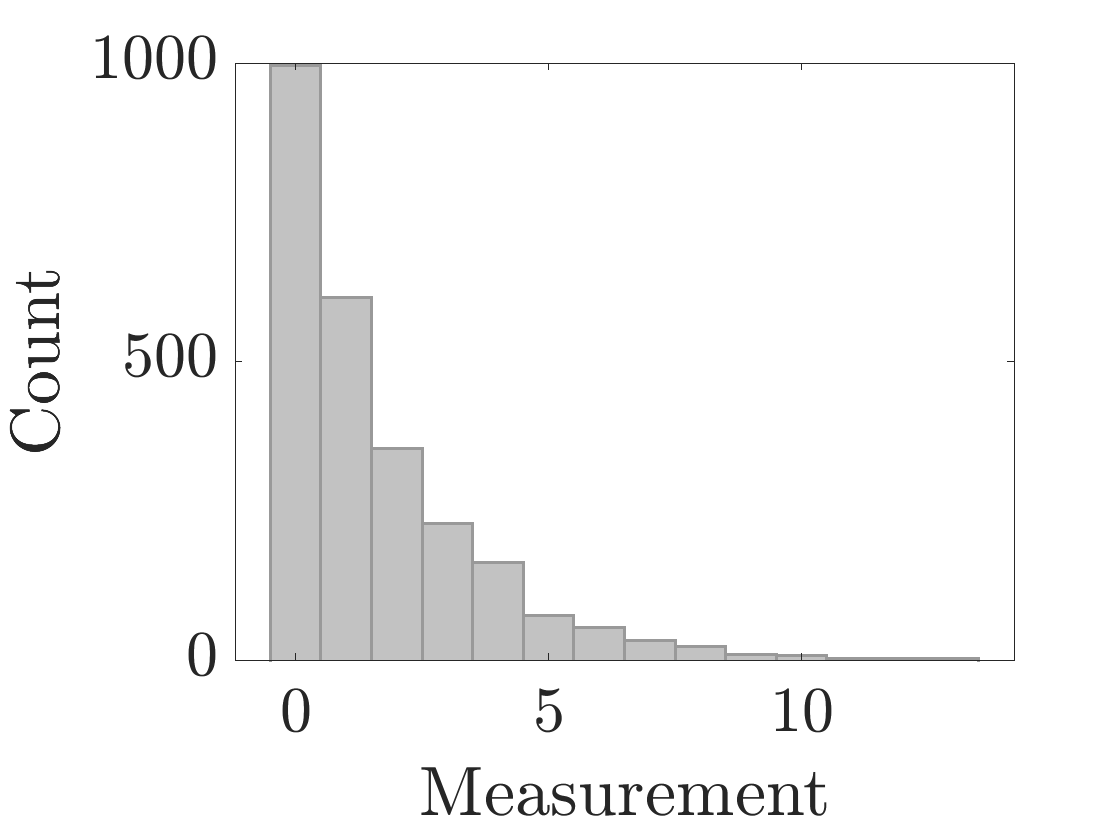}\\
\figurecaptionfont (b) Dose $d = 4000$.
\end{minipage}
\caption{Distribution of measurements for a low-dose and a higher-dose experiment.}
\label{histogram}
\end{figure}

We initialize all considered algorithms using the power method proposed in \cite{candes2015phase}. For each dose we repeat the experiments twenty times and show averages of the reconstruction results.

\subsubsection{Gradient descent with Poisson log-likelihood}

First, we compare the performance of the gradient descent algorithm using the Poisson log-likelihood loss \eqref{regularized_Poisson_loss} for different regularization parameters $\varepsilon$. The respective algorithm is named `Poisson flow'. 

\begin{figure}[!ht]
\centering
\includegraphics[width=0.85\textwidth]{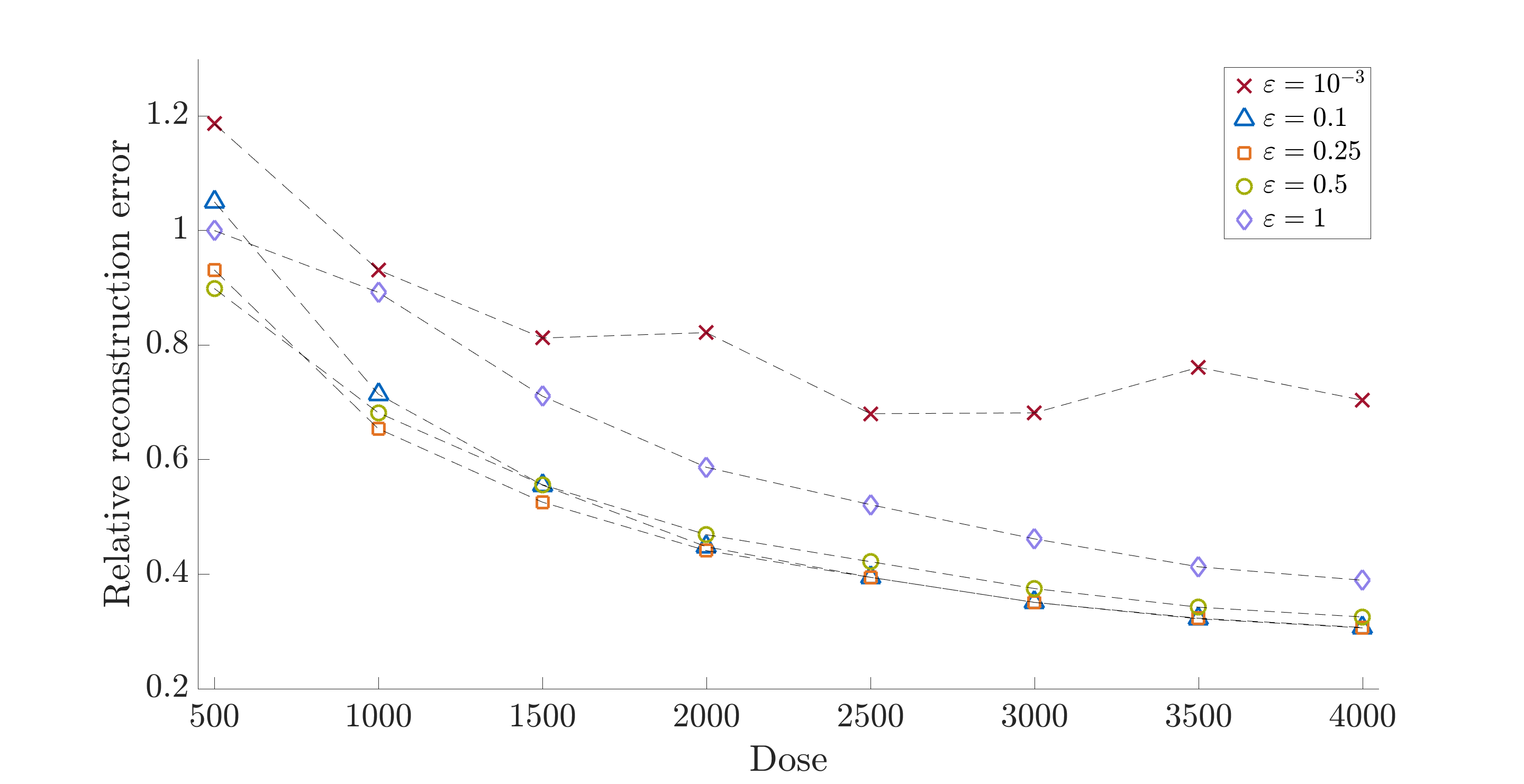}
\caption{Performance of the Poisson flow with regularization parameter $\varepsilon \in \left\{10^{-3}, 0.1, 0.25, 0.5,1 \right\} $.}
\label{fig: rel_err_poisson}
\end{figure}

\Cref{fig: rel_err_poisson} depicts the performance of the discussed algorithms for this experiment in terms of the relative reconstruction error $\min_{\theta \in \left[0,2 \pi\right]} \|x - e^{i\theta} \tilde{x}\|_{2} \ \big/ \ \| x \|_{2}$, with ground-truth $x$ and approximate reconstruction $\tilde{x}$. This experiment shows that the performance of the algorithm clearly depends on the choice of the regularization parameter $\varepsilon$. While in a higher-dose experiment the algorithm is not very sensible to the parameter selection, a low-dose experiment requires knowledge about a good choice for $\varepsilon$.

We would like to comment that, alternatively to \eqref{regularized_Poisson_loss}, we could also work with the loss 
\begin{equation*} 
\mathcal{L}(z) = \sum_{i=1}^{m} \left\vert\left\langle a_{i},z\right\rangle \right\vert^{2} - (y_{i} + \varepsilon) \log\left(\left\vert\left\langle a_{i},z\right\rangle \right\vert^{2} + \varepsilon\right).
\end{equation*}
The corresponding method can be understood as unbiased since this loss is minimized by $z$ satisfying $\left\vert\left\langle a_{i},z\right\rangle \right\vert^{2} = y_{i}, \ i = 1,\ldots, m$. Hence, this method would be less sensible to the selection of $\varepsilon$. However, we found in our numerical analysis that the corresponding algorithm performs worse than the algorithm using loss \eqref{regularized_Poisson_loss} with a good parameter $\varepsilon$. Moreover, the problematic that the recommended step size is parameter dependent remains.

\subsubsection{Gradient descent with Gaussian log-likelihood after variance stabilization}

Further, we compare the Poisson flow with gradient descent algorithms using Gaussian log-likelihood losses, with and without variance stabilization. We denote with `Wirtinger flow', as common in the literature, the algorithm using the loss \eqref{loss_gauss} with constant variance $\sigma_{i}^2 = 1/4$. `Amplitude flow' means the gradient descent algorithm for loss \eqref{amplitude_loss}. The algorithm we label here as `Flow with optimized variance stabilization' uses the loss function \eqref{0_adaption} with $c_1 = 0.12, \ c_2 = 0.27$ and $C = \sqrt{y_{i} + c_1} + \sqrt{y_{i}+c_2}$ and algorithm `Flow with optimized variance stabilization without adaption for zeros' uses loss \eqref{loss_averaging_transform} with the same parameters. 

\begin{figure}[!h]
\centering
\includegraphics[width=0.85\textwidth]{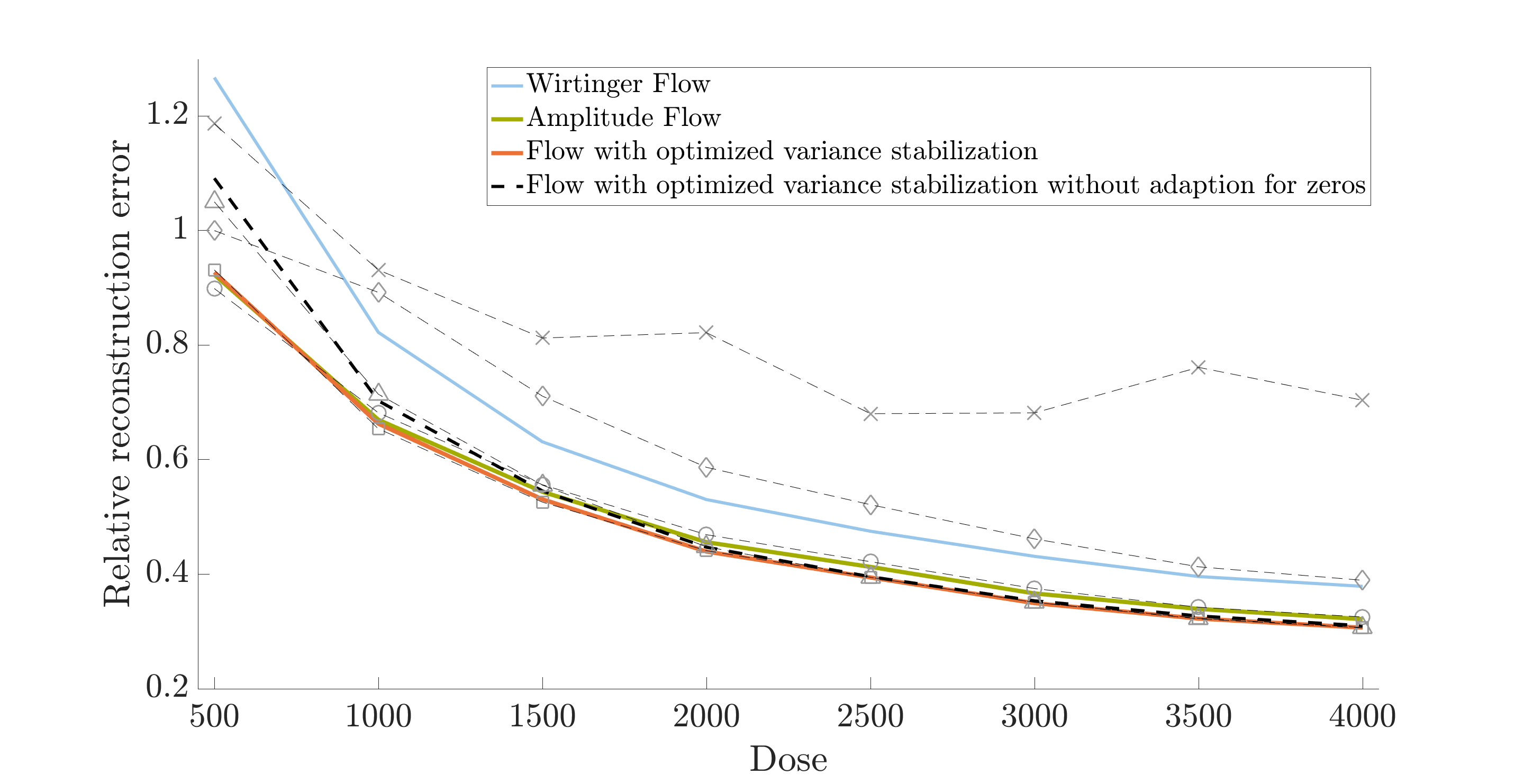}
\caption{Performance of the gradient descent algorithms for the different discussed loss functions. The gray lines correspond to the Poisson flow as in \Cref{fig: rel_err_poisson}.}
\label{fig: rel_err_all}
\end{figure}

The plot of the relative reconstruction errors for different doses in \Cref{fig: rel_err_all} indicates that the algorithm using the suggested loss function with the optimized variance-stabilizing transform including the Poisson log-likelihood adaption to zero measurements performs comparably to the Poisson flow for a close to optimal choice of regularization parameter. If we are aware of an optimal parameter selection, it is reasonable to apply the Poisson flow. However, if no decision rule is available, it is advisable to use a form of regularized amplitude loss function with adaption. 

We note that the adaption \eqref{0_adaption} of the loss for zero measurements is essential for a low-dose experiment dominated by zero measurements. In the amplitude flow using the amplitude loss function \eqref{amplitude_loss} we did not include such an adaption. In this special case, this adaption is implicitly contained in the method as the amplitude loss for zero measurements corresponds to the exact Poisson log-likelihood loss for zero measurements.

Aside from that, we recognize that for a very low-dose experiment the amplitude flow works just as well as the flow using the variance stabilization improved for low-count experiments. This can be explained by the fact that for a distribution parameter close to $0$ the square-root transform performs best in terms of variance stabilization, see \Cref{fig:variance_stabilizing_transforms}.

Interpreting the reconstruction results, we have to note that we are in the low-dose regime and cannot expect very small reconstruction errors without much redundancy in the data, i.e., a large amount of measurements. Furthermore, we have to mind that for drawing more measurements, the dose per measurement needs to be reduced accordingly in a low-dose experiment because more measurements cause more damage.

\subsection{Discussion of the step size selection}

For the Wirtinger flow algorithm using the Gaussian log-likelihood loss we do not find a constant step size that guarantees a decrease of the loss function in each iteration as for the other loss functions discussed in this paper.
We can choose an iteratively adapted step size, based on an analysis similar to the proof ideas used over all this paper, where it is not possible to get rid of the dependence on the iterates. However, this implementation is computationally expensive and makes the choice of this loss function less attractive.
For the experiments, we used the step size proposed in \cite{chen2023generalized} as this is computationally more efficient.

For all other algorithms, we use the step sizes proposed by the respective theoretical results.

\cite{li2022poisson} suggest to use a step size based on an approximation of the Hessian using observed Fisher information for the Poisson log-likelihood loss. This step size is claimed to accelerate the convergence as it is larger than the step size involving the inverse of the factor $1+\frac{y_{i}}{8 \varepsilon}$. However, in comparison to the step size derived in \Cref{theorem_convergence_poisson}, the Fisher information based step size solves a line search in each iteration, which results in computational expense. 

Here, we face a trade-off between having a constant step size, independent of the approximate obtained in each iteration, that guarantees convergence or having a step size rule that can be made independent of $\varepsilon$ by performing a step size optimization in each iteration. Due to the computational advantage, our interest lies rather in the constant step size.

We found in our numerical experiments that for larger regularization parameters $\varepsilon$ the step size proposed by \cite{li2022poisson}, as contrasted with our step size rule, does not guarantee descent of the loss in each iteration. While for $\varepsilon \rightarrow 0$ our step size rule becomes impractical, the iteration dependent step size becomes more useful as it does not decrease on the order of $\varepsilon$. On the other hand, our experiments showed that it is not reasonable to work with an extremely small regularization parameter $\varepsilon$. 

We conclude that using the constant step size in combination with $\varepsilon$ large enough seems to be reasonable. Other than that, we would like to emphasize again that this trade-off problematic can be avoided when using the Gaussian log-likelihood with a good variance stabilization as suggested here.

\section{Conclusion} \label{conclusion}

In this paper, we studied gradient descent algorithms for suitable regularizations and approximations of the Poisson log-likelihood problem. Motivated by the nature of low-dose imaging experiments, we designed a method that allows for improved treatment of zero measurements and other small counts. For all discussed algorithms, we provided a convergence analysis including a step size rule.

In terms of applying such methods to real-world data, a next step is to incorporate a suitable regularization term in the optimization problem, for example of the type discussed in \cite{le2007variational} or \cite{loock2014phase}. The optimal choice for the regularization term depends highly on the type of object under consideration. It is future work to find a reasonable problem formulation for objects such as, for example, viruses that shall be imaged using low-dose illumination. A possible follow-up work on this paper would be to extend the convergence analysis to algorithms involving such regularization attempts.

\vspace{6mm}

\section*{Declarations}

\bmhead{Funding}
The authors acknowledge support by the Helmholtz Association under the contracts No.~ZT-I-0025 (Ptychography 4.0), No.~ZT-I-PF-4-018 (AsoftXm), No.~ZT-I-PF-5-28 (EDARTI), No.~ZT-I-PF-4-024 (BRLEMMM). Patricia Römer was further financially supported by the German Research Foundation (DFG) grants KR 4512/1-1 and KR 4512/2-2.

\bibliography{sn-bibliography.bbl} 


\end{document}